\documentclass{amsart}
\usepackage{amssymb,amscd}

 \newtheorem{thm}{Theorem}[section]
 \newtheorem{cor}[thm]{Corollary}
 \newtheorem{lemma}[thm]{Lemma}
 \newtheorem{prop}[thm]{Proposition}

 \theoremstyle{definition}

 \newtheorem{defn}[thm]{Definition}
 \newtheorem{exmp}[thm]{Example}

 \theoremstyle{remark}

 \newtheorem{rem}{Remark}

 \newtheorem{ack}{Acknowledgment}

 \numberwithin{equation}{section}

\newcommand{\TT}{\text{$\mathcal{T}$}}
\newcommand{\BB}{\text{$\mathcal{B}$}}

\newcommand{\OO}{\text{$\mathcal{O}$}}
\newcommand{\UU}{\text{$\mathcal{U}$}}
\newcommand{\FF}{\text{$\mathcal{F}$}}
\newcommand{\GG}{\text{$\mathcal{G}$}}
\newcommand{\VV}{\text{$\mathcal{V}$}}
\newcommand{\WW}{\text{$\mathcal{W}$}}
\newcommand{\CC}{\text{$\mathcal{C}$}}

\newcommand{\LB}{\text{$\Lambda$}}
\newcommand{\sg}{\text{$\sigma$}}

\newcommand{\infim}{\operatorname{inf}}
\newcommand{\minim}{\operatorname{min}}
\newcommand{\Cat}{\operatorname{Cat}}
\newcommand{\dm}{\operatorname{dim}}
\newcommand{\str}{\operatorname{sat}}

\newcommand{\intr}{\operatorname{int}}
\newcommand{\Crit}{\operatorname{Crit}}
\newcommand{\id}{\operatorname{id}}
\newcommand{\Nil}{\operatorname{Nil}}
\newcommand{\Ker}{\operatorname{Ker}}
\newcommand{\Image}{\operatorname{Im}}
\newcommand{\Meas}{\operatorname{Meas}}
\newcommand{\MT}{\operatorname{MT}}

        \newcommand{\field}[1]{\text{$\mathbb{#1}$}}
        \newcommand{\N}{\field{N}}
        \newcommand{\Z}{\field{Z}}
        \newcommand{\Q}{\field{Q}}
        \newcommand{\R}{\field{R}}

\newdimen\theight
\def\TeXref#1{%
             \leavevmode\vadjust{\setbox0=\hbox{{\tt
                     \quad\quad  {\small \textrm #1}}}%
             \theight=\ht0
             \advance\theight by \lineskip
             \kern -\theight \vbox to
             \theight{\rightline{\rlap{\box0}}%
             \vss}%
             }}%

\begin{document}

\title{Measurable versions of the LS category on laminations}

\author{Carlos Meni\~no Cot\'on}

\address{Departamento de Xeometr\'{\i}a e Topolox\'{\i}a\\
         Facultade de Matem\'aticas\\
         Universidade de Santiago de Compostela\\
         15782 Santiago de Compostela}

\email{carlos.meninho@gmail.com}

\thanks{Supported by ``Ministerio de Ciencia
e Innovaci\'{o}n'', Spain (grants FPU and MTM2008-02640)}


\begin{abstract}
We give two new versions of the LS category for the set-up of measurable laminations defined by Berm\'udez. Both of these versions must be considered as ``tangential categories''. The first one, simply called (LS) category, is the direct analogue for measurable laminations of the tangential category of (topological) laminations introduced by Colman Vale and Mac\'ias Virg\'os. For the measurable lamination that underlies any lamination, our measurable tangential category is a lower bound of the tangential category. The second version, called the $\Lambda$-category, depends on the choice of a transverse invariant measure $\Lambda$. We show that both of these ``tangential categories'' satisfy appropriate versions of some well known properties of the classical category: the homotopy invariance, a dimensional upper bound, a cohomological lower bound (cup length), and an upper bound given by the critical points of a smooth function.
\end{abstract}

\maketitle

\tableofcontents

\section{Introduction}
The LS category is a homotopy invariant defined as the minimum number of
open subsets contractible within a topological space needed to
cover it. Many variations of this invariant has been also defined. In
particular, H.~Colman Vale and E.~Mac\'ias Virg\'os introduced a tangential version for
foliations, where they use leafwise contractions to transversals
\cite{HellenColman,Macias}.

In this paper, we consider the set-up of measurable laminations defined by Berm\'udez \cite{Bermudez,Bermudez-Hector}. They are ``laminations of measurable spaces'' in an obvious sense; there is a leaf topology, but the ambient space topology is replaced by a weaker ambient measurable structure. In measurable laminations, we use topological (or differentiable) terms to refer to the leaf topology, and measurable terms to refer to the ambient measurable structure; for instance, a measurable open set is a set that is leafwise open and measurable in the ambient space. Then we introduce two versions of the tangential category for a measurable lamination $\FF$ on a space $X$, which are new even $\FF$ is the underlying measurable lamination of a (topological) lamination. The first one, simply called its (LS) category, is a direct adaptation to measurable laminations of the tangential (or even the usual) category. The second one is called the $\Lambda$-category because it involves a transverse invariant measure $\Lambda$, whose existence is a restriction.

Let $\FF$ be a measurable lamination on a measurable space $X$. The category of $\FF$ is the minimum number of categorical measurable open sets needed to cover $X$, where ``categorical'' means that there exists a measurable continuous deformation to some transversal. Obviously, if $\FF$ is the measurable lamination that underlies any lamination, our category is a lower bound of the tangential category, and sometimes it is easier to study (measurability allows more freedom to make constructions than continuity).

To define the $\Lambda$-category of $\FF$ for a transverse invariant measure $\Lambda$, we can also take coverings of $X$ by categorical measurable open sets, but now $\Lambda$ is used to ``count'' them: we consider the sum of the measures of the transversals resulting from their deformations. The infimum of all those possible measures is the $\Lambda$-category of $\FF$.

For these two new ``tangential categories'', we prove appropriate
versions of some classical results about LS category, like their
homotopy invariance, a cohomological lower bound given by the cup
length, a dimensional upper bound (adapted to the tangential
category of foliations by W. Singhof and E. Vogt
\cite{Vogt-Singhof}), or a lower bound given the number of
critical points of a smooth function. More precisely, it is proved
that the category is less or equal than the number of critical
sets of a differentiable function, where the critical sets are
defined by using the leafwise gradient flow of the function. This
improves even the classical result because a smooth function on a
manifold may have less critical sets than critical points. On the
other hand, the $\Lambda$-category is less or equal than the
measure of the set of leafwise critical points; i.e., the critical
points of the restrictions of the function to the leaves.
Following the work of J. Schwartz \cite{Schwartz}, these relations
with critical points are obtained in the setting of measurable
Hilbert laminations because of their possible applications to
tangential variational
problems~\cite{Ballman,Lusternik-Schnirelmann}.

\section{MT-spaces and measurable laminations}

A {\em measurable topological space\/}, or {\em MT-space\/}, is
a set $X$ equipped with a \sg-algebra and a topology. Usually, measure
theoretic concepts will refer to the \sg-algebra of $X$, and
topological concepts will refer to its topology; in
general, the \sg-algebra is different from the Borel \sg-algebra
induced by the topology. An {\em MT-map\/} between MT-spaces is a
measurable continuous map. An {\em MT-isomorphism} is a map between MT-spaces that is a measurable isomorphism and a homeomorphism. Trivial examples of MT-spaces are topological spaces with their Borel
$\sg$-algebras, and measurable spaces with the discrete topology.

Let $X$ and $Y$ be MT-spaces. Suppose that there exists a measurable
embedding $i:X\to Y$ that maps measurable sets to measurable sets. Then $X$ is called an {\em MT-subspace\/} of $Y$.  Notice that, if $X$ and $Y$ are standard\footnote{Recall that a {\em Polish space\/} is a completely metrizable and separable
topological space, and a {\em standard Borel space\/} is a measurable
space isomorphic to a Borel subset of a Polish space.}, the measurability of $i$ means that it maps Borel sets to Borel sets
\cite{Srivastava}. The product $X\times Y$ is an MT-space too with the product topology and the $\sg$-algebra generated by
products of measurable sets of $X$ and $Y$.

Let $R$ be an equivalence relation on an MT-space $X$. The quotient set $X/R$ becomes an MT-space with the quotient topology and the $\sg$-algebra generated by the projections of measurable saturated sets of $X$; it can be called the {\em quotient\/} MT-space.

Let $T$ be a standard Borel space and let $P$ be a Polish space. Let $P\times T$
be endowed with the structure of MT-space defined by the
\sg-algebra generated by products of Borel subsets of $T$ and
Borel subsets of $P$, and the product of the discrete topology on
$T$ and the topology of $P$.

A {\em measurable chart} on an MT-space $X$ is an MT-isomorphism
$\varphi:U\to B\times T$, where $U$ is open and measurable in $X$
and $B$ is an open ball in $\R^n$, the simpler notation
$(U,\varphi)$ may be used for such a chart. The sets
$\varphi^{-1}(B\times\{\ast\})$ are called {\em plaques\/} of
$\varphi$, and the sets $\varphi^{-1}(\{\ast\}\times T)$ are
called {\em transversals\/} associated to $\varphi$. A {\em
measurable atlas\/} on $X$ is a countable family of foliated
measurable charts whose domains cover $X$. A {\em measurable
lamination\/} of {\em dimension\/} $n$ is an MT-space that admits
a countable measurable atlas consisting of charts
$\varphi_i:U_i\to B_i\times T_i$ ($i\in\N$), where each $B_i$ is an open ball in $\R^n$. Since we have used a
countable atlas, the ambient space is also a standard space. The
connected components of $X$ are called its {\em leaves\/}; they
are second countable connected manifolds, which may not be
Hausdorff. The {\em saturation\/} of a set $B$, denoted by
$\str(B)$, is the union of leaves that meet $B$. The saturation of
a measurable set is measurable since the measurable atlas is
countable.

The typical example of measurable lamination to keep in mind is given by
any lamination (or foliation), by considering the Borel \sg-algebra of its ambient space and its leaf
topology; in this case, ``leaves'' and ``saturations'' have the usual meanings.

In the setting of measurable laminations, the concept of $C^r$ tangential structure cannot be defined as a maximal atlas with (tangentially) $C^r$ changes of coordinates because the atlases are required to be countable, but we can proceed as follows. A measurable atlas is said to be (tangentially) $C^r$ if its coordinate changes are (tangentially) $C^r$. Then a $C^r$ structure is an equivalence class of $C^r$ measurable atlases, where two $C^r$ measurable atlases are equivalent if their union is a $C^r$ measurable atlas.

A measurable subset $T\subset X$  is called a {\em transversal\/}
if its intersection with each leaf is countable \cite{Heitsch-Lazarov}; these are
slightly more general than the transversals of \cite{Bermudez} that consist of isolated points, this kind of transversals are said to be {\em isolated\/}. Let
$\TT(X)$ be the family of transversals of $X$. This set is closed
under countable unions and intersections, but it is not a
\sg-algebra. A transversal meeting all leaves is called
{\em complete\/}.

A {\em measurable holonomy transformation\/} is a measurable
isomorphism $\gamma:T\to T'$, for $T, T' \in \TT(X)$, which maps
each point to a point in the same leaf. A {\em transverse
invariant measure\/} on $X$ is a \sg-additive map, $\LB:\TT(X)\to
[0,\infty]$, invariant by measurable holonomy transformations. The
classical definition of transverse invariant measure of a lamination is a measure on topological
transversals invariant by holonomy transformations (see e.g. \cite{Candel-Conlon}); both notions of transverse invariant measures agree
in this case \cite{Connes}. In this paper, we always consider \sg-finite
measures.

\begin{rem}\label{r:functorO}
There exists an underlying or forgetful functor $\OO$ from the category of foliated spaces and foliated maps (those that map leaves to leaves) to the category of measurable laminations and MT-maps.
\end{rem}

The leaf space of a measurable lamination $\FF$ is denoted
by $X/\FF$. Also, for any MT-subspace $Y\subset X$, $Y/\FF$ will denote the quotient space of $Y$ with the restricted relation (``lying in the same leaf of $\FF$'').

Our principal tools in this setting are the following two results (see e.g. \cite{Srivastava}).

\begin{prop}[Lusin]\label{p:Lusin1}
Let $X$ and $Y$ be standard Borel spaces and $f:X\to Y$ a
measurable map with countable fibers. Then $f(X)$ is Borel in $Y$
and there exists a measurable section $s:f(X)\to X$ of $f$. In
particular, if $f$ is injective, then $s$ is a Borel isomorphism.
Moreover there exists a countable Borel partition, $X=\bigcup_i
X_i$, such that each $f|_{X_i}$ injective.
\end{prop}

\begin{thm}[Kunugui, Novikov]\label{t:Srivastava}
Let $\{V_n\}_{n\in\N}$ be a countable base for a Polish space $P$.
Let $B\subset P\times T$ be a Borel set such that
$B\cap(P\times\{t\})$ is open for every $t\in T$. Then there
exists a sequence $\{B_n\}_{n\in\N}$ of Borel sets of $T$ such
that
$$B=\bigcup_n (V_n\times B_n)\;.$$
\end{thm}

  Every measurable open MT-subspace $U$ of a measurable lamination $\FF$ is a measurable lamination (by the previous theorem); the notation $\FF_U$ will be used in this case. By restriction, any transverse invariant measure $\Lambda$ of $\FF$ induces a transverse invariant measure of $\FF_U$, which will be denoted by $\Lambda_U$.

\begin{lemma}[\cite{Menino1}]\label{l:cocycle}
Let $\varphi_i:U_i\to B_i\times T_i$ and $\varphi_j:U_j\to
B_j\times T_j$ be measurable charts of $X$. There exists a
sequence of Borel sets of $T_i$, $\{S_n\}_{n\in\N}$, and a base of
$B_i$, $\{V_n\}_{n\in\N}$, such that $\varphi_i(U_i\cap
U_j)=\bigcup_n(V_n\times S_n)$ and
$\varphi_j\circ\varphi_i^{-1}(x,t)= (g_{ijn}(x,t), f_{ijn}(t))$
for $(x, t)\in V_n\times S_n$, where each $f_{ijn}$ is a Borel
isomorphism and each $g_{ijn}$ is an MT-map.
\end{lemma}

\begin{rem}
The previous lemma is also true replacing the open balls $B_i$ of an euclidean space by any connected and locally
connected Polish space.
\end{rem}

\begin{defn}
A foliated measurable atlas $\UU$ is called {\em regular\/} if:
  \begin{itemize}

    \item[(i)] for all $(U,\varphi)\in\UU$, there exists another measurable foliated chart $(W,\psi)$ such that the closure of each plaque in $U$ is compact, $\overline{U}\subset W$ and $\varphi=\psi|_U$; and,

    \item[(ii)] for all $(U_1,\varphi_1),(U_2,\varphi_2)\in\mathcal{U}$, each plaque of $(U_1,\varphi_1)$ meets at most one plaque of $(U_2,\varphi_2)$.

  \end{itemize}
\end{defn}

Observe that, if $\UU$ is a regular measurable atlas, then $\overline{U}$ is measurable for all $(U,\varphi)\in\UU$.

This definition of regular measurable atlas is weaker than the usual one for laminations (see
e.g. \cite{Candel-Conlon}): the locally finite condition
does not make sense for measurable laminations since there is no
ambient topology. The following result follows from
Lemma~\ref{l:cocycle}.

\begin{cor}
If a measurable lamination has a foliated measurable atlas such that each chart meets a finite number of charts, then it admits a regular measurable foliated atlas.
\end{cor}

From now on, we consider only measurable laminations that
admit regular measurable foliated atlases.

\begin{exmp}[Measurable suspensions \cite{Bermudez}]
Let $P$ be a connected, locally path connected and semi-locally $1$-connected Polish space, and let $S$
be a standard space. Let $\Meas (S)$ denote the group of
measurable transformations of $S$. Let $h : \pi_{1}(P, x_{0})\to \Meas(S)$ be a homomorphism. Let
$\widetilde{P}$ the universal covering of $P$ and consider
the action of $\pi_{1}(P, x_{0})$ on the MT-space
$\widetilde{P}\times S$ given by $g\cdot(x, y) = (xg^{-1}, h(g)(y))$. The corresponding quotient MT-space, $\widetilde{P}\times_h S$, is called the {\em measurable suspension\/} of
$h$. If $P$ is a manifold, then $\widetilde{P}\times_h S$ is a measurable lamination, $\{\ast\}\times S$
is a complete transversal, and its leaves are covering spaces of $P$.
\end{exmp}

\section{Category of measurable laminations}

A measurable lamination $(X,\FF)$ induces a foliated
measurable structure $\FF_U$ in each measurable open set $U$ (by
Theorem~\ref{t:Srivastava}). The space $U\times\R$ admits an
obvious foliated structure $\FF_{U\times\R}$ whose leaves are
products of leaves of $\FF_U$ and $\R$. Let $(Y,\GG)$ be another measurable lamination. An MT-map $H:\FF_{U\times\R}\to \GG$ is called a {\em \upn{(}measurable\upn{)}
homotopy\/}, and it is said that the maps $H(\cdot,0)$ and $H(\cdot,1)$ are {\em \upn{(}measurably\upn{)} homotopic\/}. We use the term {\em
\upn{(}measurable\upn{)} deformation\/} when $\GG=\FF$ and $H( -
,0)$ is the inclusion map of $U$. A deformation such that $H( - ,1)$ is constant on the leaves of
$\FF_U$ is called a {\em \upn{(}measurable\upn{)} contraction\/} or an \FF-{\em
contraction\/}; in this case, $U$ is called a {\em categorical\/} or \FF-{\em categorical\/} measurable open set.

The {\em\upn{(}LS\upn{)} category\/} is the lowest number of
categorical measurable open sets that cover the measurable lamination. On one leaf foliations, this definition agree
with the classical category. The category of \FF\ is
denoted by $\Cat(\FF)$. It is clear that it is a homotopy invariant.

All of these definitions have obvious versions for laminations by using ambient space continuity instead of measurability, obtaining the so called tangential category~\cite{HellenColman,Macias}.

\begin{defn}
Let $U\subset (X,\FF)$ be a measurable open subset. Define the {\em relative category\/} of $U$, $\Cat(U,\FF)$, as the minimum number of categorical measurable open sets  in $(X,\FF)$ that form a cover of $U$.
\end{defn}

\begin{rem}
Clearly, $\Cat(U,\FF)\leq\Cat(\FF_{U})$.
\end{rem}

\begin{prop}[Subadditivity of the relative category]\label{p:subaddtivity of the relative category}
Let $\{U_i\}_{i\in\N}$ be a countable family of measurable open subsets of $X$. Then
\[
  \Cat\left(\bigcup_iU_i,\FF\right)\leq\sum_i\Cat(U_i,\FF)\;.
\]
\end{prop}

\section{$\Lambda$-category}

\begin{lemma}[Kallman \cite{Kallman}]\label{l:Kallman}
  Let $P\times T$ be a product of a Polish space $P$ with a standard space $T$, and let $\pi:P\times T\to T$ denote the second factor projection. Let $B\subset P\times T$ be a measurable subset such that $B\cap(P\times\{t\})$ is \sg-compact for all $t\in T$. Then $\pi(B)$ is measurable.
\end{lemma}

\begin{prop}\label{p:measurability}
For any measurable open set $U$ and any tangential deformation $H$, the set $H(U\times\{1\})$ is measurable.
\end{prop}

\begin{proof}
By the Kunugui-Novikov's theorem (Theorem~\ref{t:Srivastava}), there exist countable families, $\{P_n\times T_n\}$ and $\{P'_n\times T'_n\}$ ($n\in\N$), and MT-embeddings $f_n:P_n\times T_n\to\FF$ and $g_n:P'_n\times T'_n\to \FF$ such that $U=\bigcup_n f_n(P_n\times T_n)$ and $H(\cdot,1)\circ f_n(P_n\times T_n)\subset g_n(P'_n\times T'_n)$. Hence we only have to prove that each $g_n^{-1}\circ H(\cdot,1)\circ f_n (P_n\times T_n)$ is measurable. Consider a topology on $T_n$ so that it is isomorphic to $[0,1]$, $\N$ or a finite set \cite{Kechris}, and let $P_n\times T_n$ be endowed with the product topology $\tau$, becoming a Polish space. Let $\pi:(P_n\times T_n,\tau)\times P'_n\times T'_n\to P'_n\times T'_n$ be the second factor projection. For each point $(x,t)\in P'_n\times T'_n$, the set $f_n^{-1}\circ H(\cdot,1)^{-1}\circ g_n (x',t')$ is \sg-compact in $(P_n\times T_n,\tau)$. By the continuity of $H(\cdot,1)\circ f_n$ (with the leaf topology), this preimage is closed on each plaque $P_n\times\{t\}$ ($t\in T_n$). On the other hand, this preimage only cuts a countable number of leaves (otherwise $H$ would not be a tangential deformation), and therefore it only cuts a countable number of plaques. Hence this preimage is also \sg-compact with the topology $\tau$. Let
  \[
    B_n=\{\,((x,t),(x',t'))\mid g_n^{-1}\circ H(\cdot,1)\circ f_n(x,t)=(x',t')\,\}\;,
  \]
which is measurable in $(P_n\times T_n,\tau)\times P'_n\times T'_n$. Then, by Lemma~\ref{l:Kallman}, $\pi(B_n)$ is measurable. But $\pi(B_n)=g_n^{-1}\circ H(\cdot,1)\circ f_n(P_n\times T_n)$.
\end{proof}

\begin{prop}[\cite{Menino1}]\label{p:coherentextension}
Let $(X,\FF)$ be a measurable lamination with a transverse invariant measure
\LB. There exists a unique Borel measure $\widetilde{\LB}$ on $X$ such
that $\widetilde{\LB}(T)=\LB(T)$ for all generalized transversal
$T$, and satisfying the following properties:
\begin{itemize}

\item[(i)] If $B$ is a measurable set so that $B\not\subset\pi^{-1}(S)$ for any transversal $S$ with $\LB(S)=0$, and $\partial(B\cap L)\neq\emptyset$ for each leaf $L$ that meets $B$, then
$\widetilde{\LB}(B)=\infty$.

\item[(ii)] If $\LB(S)=0$ for some $S\in\mathcal{B}$, then
$\widetilde{\LB}(\str(S))=0$.

\item[(iii)] If $B$ is a measurable set so that $\LB(S)=\infty$ for all transversal $S$ with $B\subset\str(S)$, then $\widetilde{\LB}(B)=\infty$.

\end{itemize}
\end{prop}

The unique extended measure $\widetilde{\LB}$, given by Proposition~\ref{p:coherentextension}, is called the {\em coherent extension\/} of \LB.

\begin{rem}
Observe that, in measurable charts of the form $P\times T$, the coherent extension can be given on measurable sets with \sg-compact intersection with the plaques as the pairing of the counting measure with the transverse invariant measure \LB\ \cite{Menino1}:
$$\widetilde{\LB}(B)=\int_T \#(B\cap (P\times\{t\}))\,d\LB(t)\;.$$
\end{rem}

Let $\LB$ be a transverse invariant measure for
\FF, and $\widetilde{\LB}$ its coherent extension. According to Proposition~\ref{p:measurability}, we can define
$$
  \tau_\LB(U)=\infim\{\,\widetilde{\LB}(H(U\times\{1\})\mid H\ \text{is a tangential deformation of}\ U\,\}\;.
$$
Then the \LB-{\em category\/} of $(\FF,\LB)$ is defined as
$$\Cat(\FF,\LB)=\infim_{\UU}\sum_{U\in\UU}\tau_\LB(U)\;,$$
where \UU\ runs in the countable coverings of $X$ by measurable open sets. The
countability condition of the coverings is needed, otherwise the
\LB-category would be directly zero when \LB\ has no atoms. If the homotopies
used in this definition are required to be $C^r$ on leaves, then the term $C^r$ \LB-{\em
category\/} is used, with the notation $\Cat^r(\FF,\LB)$.

The category and the \LB-category are connected by
the following result.

\begin{prop}\label{p:posit}
Let $(X,\FF,\LB)$ be a measurable lamination with a transverse
invariant measure, and let $U$ be a measurable open set in $X$. If
$\tau_\LB(U)<\infty$, then there exists a categorical measurable open set $U'\subset U$ such that $\widetilde{\LB}(U\setminus U')=0$.
\end{prop}

\begin{proof}
There exists a tangential deformation of $U$ such that
$\widetilde{\LB}(H(U\times\{1\}))<\infty$. This means, by the conditions of the coherent extension, that
$H(U\times\{1\})=B\cup T$, where $B$ is a Borel set with
$\widetilde{\LB}(B)=0$ and $T$ is a transversal of \FF. The set $U'=H(\cdot,1)^{-1}(T)$ satisfies the required
conditions.
\end{proof}

\begin{defn}
Let $(X,\FF,\LB)$ be a foliated measurable space with a transverse
invariant measure. A {\em null-transverse\/} set is a measurable
set $B$ such that $\widetilde{\LB}(B)=0$.
\end{defn}

The following propositions are elementary.

\begin{prop}\label{p:nullset}
Let $(X,\FF,\LB)$ be a measurable lamination with a transverse
invariant measure, and let $B$ be a null-transverse set. Then
$\Cat(\FF,\LB)$ can be computed by using only coverings of
$X\setminus \str(B)$. If $B$ is saturated, then $\Cat(\FF,\LB)=\Cat(\FF_{X\setminus B},\LB_{X\setminus B})$.
\end{prop}

\begin{prop}
Let $T$ be a measurable transversal which meets each leaf at
most in one point. Then $\LB(T)\leq \Cat(\FF,\LB)$.
\end{prop}

\begin{prop}\label{p:boundsusp}
Let $(X,\FF,\LB)$ be a measurable lamination with a transverse
invariant measure. If $(X,\FF)$ is a
measurable suspension $\widetilde{M}\times_h S$, then $\Cat(\FF,\LB)\leq \Cat(M)\cdot \LB(S)$.
\end{prop}

\begin{prop}\label{p:products}
For a manifold $M$ and a standard Borel space $T$, let $M\times T$
be foliated as a product. Then $\Cat(M\times T,\LB)=\Cat(M)\cdot\LB(T)$ for every measure \LB\ on $T$, considered as an
invariant measure of $M\times T$.
\end{prop}

\begin{prop}\label{p:catsubspaces}
Let $\{U_n\}$ \upn{(}$n\in\N$\upn{)} be a covering by saturated measurable open
sets of $(X,\FF,\LB)$. Then $\Cat(\FF,\LB)\leq\sum_n\Cat(\FF_{U_n},\LB_{U_n})$.
Here, the equality holds if $\{U_n\}$ is a partition.
\end{prop}

\begin{defn}
Let $U\subset (X,\FF)$ be a measurable open subset. The {\em relative $\Lambda$-category\/} of $U$ is defined by
\[
\Cat(U,\FF,\LB)=\inf_{\UU}\sum_{V\in\UU}\tau_\LB(V)\;,
\]
where $\UU$ runs in the family of countable measurable open coverings of $U$.
\end{defn}

\begin{rem}
Let us emphasize that, in the above definition, $\tau_\LB(V)$ is defined by using measurable tangential homotopies deforming $V$ in the ambient space. Clearly, $\Cat(U,\FF,\LB)\leq\Cat(\FF_{U},\LB)$.
\end{rem}

\begin{prop}[Subadditivity of the relative \LB-category]\label{p:subaddtivity of the relative LB-category}
Let $\{U_i\}_{i\in\N}$ be a countable family of measurable open subsets of $X$. Then
\[
  \Cat\left(\bigcup_iU_i,\FF,\LB\right)\leq\sum_i\Cat(U_i,\FF,\LB)\;.
\]
\end{prop}

\section{Homotopy invariance of the $\Lambda$-category}

Let $\FF$ and $\GG$ be measurable laminations. An MT-homotopy equivalence $h$ from \FF\ to \GG\ induces a canonical
bijection between the sets of transverse invariant measures on \GG\ and \FF, which is defined as follows. Let
$T$ be a complete transversal of \FF. Obviously $h|_T$ has
countable fibers. By Proposition \ref{p:Lusin1}, $h(T)$ is a
transversal of \GG. In fact, there exists a countable measurable
partition, $T=\bigcup_i T_i$, so that each $h|_{T_i}$ is
injective. Define $h^*\LB(T)=\sum_i\LB(h(T_i))$. From now on, we any MT-homotopy equivalence between measurable laminations with transverse invariant measures is assumed to be compatible with the measures in the above sense.

\begin{lemma}\label{l:invariance1}
Let $(X,\FF,\LB)$ and $(Y,\GG,\Delta)$ be foliated measurable
spaces with transverse invariant measures, and let $h:(X,\LB)\to (Y,\Delta)$ be a
measurable homotopy equivalence.
Then, for all measurable set $K\subset X$ with \sg-compact intersections with
the leaves, $h(K)$ is measurable and
$\widetilde{\Delta}(h(K))\leq\widetilde{\LB}(K)$.
\end{lemma}
\begin{proof}
The fact that $h(K)$ is measurable is a consequence of Proposition~\ref{p:measurability}. Notice that finite intersections of \sg-compact sets are \sg-compact.
Let $\UU=\{(U_n,\varphi_n)\}$ and
$\VV=\{(V_n,\psi_n)\}$ ($n\in\N$) be foliated measurable atlases for
$\FF$ and $\GG$, respectively. Observe that there exists a foliated
measurable atlas $\WW=\{(W_n,\phi_n)\}$ on $X$ satisfying
the following conditions:
\begin{enumerate}
\item[(a)] For each $n\in\N$, there exists some $k(n)\in\N$ such
that $h(W_n)\subset V_{k(n)}$.

\item[(b)] The map $h$ induces an injective map from the family of
plaques of $W_n$ to the family of plaques of $V_{k(n)}$; i.e., each plaque of $V_{k(n)}$ contains at most the image by
$h$ of one plaque of $W_n$.
\end{enumerate}
This atlas can be easily obtained by using
Theorem~~\ref{t:Srivastava} and Proposition \ref{p:Lusin1}.

The maps $\psi_{k(m)}\circ h\circ \phi_m^{-1}:P_m\times T_m\to
P'_{k(m)}\times T'_{k(m)}$ are injective in the set of plaques in the sense of
(b). Clearly, $\LB(D)=\Delta(h(D))$ for all Borel sets $D\subset
T_m$. For every $t\in T_m$, let $P^t$ be the plaque of $P'_{k(m)}\times T'_{k(m)}$ that
contains $\psi_{k(m)}\circ h\circ \phi_m^{-1}(P_m\times\{t\})$. We have
$$\# (h(K)\cap P^t)\leq\# (K\cap(P\times\{t\}))$$
for every $t\in T_m$. Hence
\begin{multline*}
  \widetilde{\Delta}(\psi_{k(m)}\circ h\circ \phi_m^{-1}(K))\\
    \begin{aligned}
      &=\int_{T'_{k(m)}}\# (\psi_{k(m)}\circ h\circ \phi_m^{-1}(K)\cap(P'\times\{t'\}))\,d\Delta(t')\\
      &=\int_{\psi_{k(m)}\circ h\circ \phi_m^{-1}(T_m)}\#( \psi_{k(m)}\circ h\circ \phi_m^{-1}(K)\cap(P\times\{t'\}))\;d\Delta(t')\\
      &=\int_{T_m}\# (\psi_{k(m)}\circ h\circ \phi_m^{-1}(K)\cap P^{t})\,d\Lambda(t)\\
      &\leq \int_{T_m}\# (K\cap(P\times\{t\}))\,d\Lambda(t)\\
      &=\widetilde{\LB}(K)\;.
    \end{aligned}
\end{multline*}
Then the lemma holds on each chart of $\WW$. Consider the family $\{B_k\}$ inductively defined by
$$
B_1=(K\cap W_1),\quad B_k=(K\cap W_k)\setminus (B_1\cup\dots\cup
B_{k-1})\quad  (k>1)\;.
$$
It is a Borel partition of $K$ whose elements are contained in the sets $W_k$
and have \sg-compact intersection with each leaf. Then
$$
\widetilde{\Delta}(h(K))=\widetilde{\Delta}\left(\bigcup_i h(B_i)\right)\leq\sum_i\widetilde{\Delta}(h(B_i))\\
\leq\sum_i\widetilde{\LB}(B_i)=\widetilde{\LB}(K)\;.\qed
$$
\renewcommand{\qed}{}
\end{proof}
\begin{prop}[The \LB-category is an MT-homotopy invariant]\label{p:invariance}
Let $(X,\FF,\LB)$ and $(Y,\GG,\Delta)$ be measurable homotopy
equivalent measurable laminations with transverse invariant measures. Then $\Cat(\FF,\LB)=\Cat(\GG,\Delta)$.
\end{prop}

\begin{proof}
Let $h:X\to Y$ be a measurable homotopy equivalence, and let $g$ be a
homotopy inverse of $h$. Let $\{U_n\}$ ($n\in\N$) be a covering of $Y$ by
measurable open sets. Then $\{h^{-1}(U_n)\}$ is a
covering of $X$ by measurable open sets. We will prove that
$\tau_{\LB}(h^{-1}(U_n))\leq\tau_\Delta(U_n)$ for all $n\in\N$.
Let $H^n$ be a measurable tangential deformation on each $U_n$,
and let $F$ be an MT-homotopy connecting the identity map and $g\circ h$.
Let
$$G=g\circ H\circ (f\times\id): h^{-1}(U)\times \R\to X\;.$$
Then $K:h^{-1}(U)\times \R\to X$, defined by
$$
K(x,t)=
\begin{cases}
F(x,2t) & \text{if $t\leq 1/2$} \\
G(x,2t-1) & \text{if $t\geq 1/2$\;,}
\end{cases}
$$
is a tangential deformation. Lemma~~\ref{l:invariance1} yields
$$\widetilde{\LB}(K(h^{-1}(U)\times\{1\}))=\widetilde{\LB}(g(H(U\times\{1\})))\leq\widetilde{\Delta}(H(U\times\{1\}))\;.$$
Hence $\tau_{\LB}(h^{-1}(U_n))\leq\tau_\Delta(U_n)$ for all
$n\in\N$. Therefore $\Cat(\FF,\LB)\leq\Cat(\GG,\Delta)$. The
inverse inequality is analogous.
\end{proof}

The above proposition has an obvious $C^r$ version.

\section{Case of laminations with compact leaves}\label{s:compactleaves}

In this section, we compute the \LB-category of (the measurable laminations underlying) a lamination $\FF$ with compact leaves on a Polish space $X$. With these conditions, there exists a countable filtration
$$\dots\subset E_\alpha\subset\dots\subset E_2\subset E_1\subset E_0=X\;,$$
such that each $E_\alpha$ is a closed saturated set, and
$E_{\alpha}\setminus E_{\alpha + 1}$ is dense in $E_\alpha$ and consists of leaves
with trivial holonomy on the foliated space $E_\alpha$. This family
is called the {\em Epstein filtration\/} of $X$ \cite{Epstein,Epstein2, Edwards-Millett-Sullivan}.

Obviously, each $E_\alpha\setminus E_{\alpha + 1}$ is a saturated measurable
open set (in the MT-structure) without holonomy. Hence, by
Proposition~\ref{p:catsubspaces},
\begin{equation}\label{compact leaves}
\Cat(\OO(\FF),\LB)=\sum_{\alpha}\Cat(\OO(\FF_{E_{\alpha-1}\setminus E_\alpha}),\LB_{E_{\alpha-1}\setminus E_\alpha})\;,
\end{equation}
where \OO\ denotes the underlying functor defined in the Remark~\ref{r:functorO}.

\begin{thm}[See e.g. \cite{Takesaki}]\label{t:takesaki}
Let $R$ be an equivalence relation on a Polish space $X$ such that
every equivalence class is a closed set in $X$. If the saturations
of open sets of $X$ are Borel, then there exists a Borel set
meeting every equivalence class in one point. If the saturations of
open sets are open, then there exists a Polish subspace meeting
every equivalence class in one point.
\end{thm}

\begin{cor}\label{p:region}
Let $(X,\FF)$ be a lamination with all leaves compact and let
$T$ be a complete transversal of \FF. Then there exists a Polish
subspace contained in $T$ meeting every leaf in one point.
\end{cor}

Let $\Gamma$ be the holonomy pseudogroup of \FF\ on $T$. Let $Q\subset T$
be a Polish subspace satisfying the statement of Theorem~\ref{t:takesaki} for the equivalence relation
defined by the orbits of $\Gamma$ on $T$. There exists a bijective map
$\pi:Q\to T/\Gamma\equiv X/\FF$ induced by the projection to the leaf space. This map is measurable with respect to the Borel \sg-algebra of $X/\FF$. By using the Epstein filtration of $(X,\FF)$, it is easy to see that $\pi$ is a Borel isomorphism, even when $X/\FF$ is not Hausdorff.
If \LB\ is a transverse invariant measure, then it induces a
measure on $Q$ since it is a Borel transversal. Hence \LB\ induces
a measure $\LB_\FF$ on $X/\FF$ via the Borel isomorphism
$\pi$. The measure $\LB_\FF$ is independent of the choice of the
Polish (Borel) sets $R$ and $T$, since all of them are equivalent by a
measurable holonomy map. By using the Epstein filtration, it easily follows that the {\em category map\/}, $\Cat:X/\FF\to \N\cup\{\infty\}$, assigning to each leaf its category, is measurable.

By~\eqref{compact leaves}, we can assume that the leaves have trivial holonomy groups to compute $\Cat(\OO(\FF),\LB)$. Moreover $X$ has a countable number of connected components, and the leaves on each connected component are homeomorphic since the holonomy groups are trivial. So, in general, $\OO(\FF)$ is MT-isomorphic to the disjoint union of
product spaces $L_n\times R_n$, where $L_n$ are the generic leaves
in each connected component and $R_n$ are Borel sets meeting
leaves in one point in these components. By
Propositions~\ref{p:products} and~\ref{p:catsubspaces},
we obtain the following corollary.

\begin{cor}\label{c:measurablecompactcategory}
Let $(X,\FF,\LB)$ be a foliated space with compact leaves endowed with
a transverse invariant measure. Then
$$\Cat(\OO(\FF),\LB)=\int_{X/\FF}\Cat(L)\;d\LB_\FF(L)\;.$$
\end{cor}

\begin{rem}
By Proposition~\ref{p:nullset}, this Corollary~\ref{c:measurablecompactcategory} applies to the case of a measurable lamination with a transverse invariant measure supported on a countable number of compact leaves.
\end{rem}

By using the same decomposition, it is easy to check that the measurable tangential category is exactly the maximum of the LS-category of the leaves, $\Cat(\OO(\FF))=\max\{\Cat(L)\mid L\in\FF\}$.

\section{Dimensional upper bound}

It is known that $\Cat(M)\leq \dim M + 1$ for any
manifold $M$ \cite{James}. This result was
adapted to the tangential category of a $C^2$ foliation $\FF$ on closed manifold by E.~Vogt
and W.~Singhoff \cite{Vogt-Singhof}, obtaining $\Cat(\FF)\leq \dim\FF +1$.
We show an adaptation to our measurable setting.

Suppose that there exists a complete Riemannian metric $g$ on the
leaves of a $C^1$ measurable lamination which varies in a
measurable way on the ambient space. Clearly, the exponential map
is measurable by the assumptions on $g$.

\begin{lemma}\label{l:injectivity radius}
The function $i:X\to \R^+\cup\{\infty\}$, defined as the
injectivity radius of the exponential map at each point, is
measurable.
\end{lemma}

\begin{proof}
Let $\UU=\{U_l\}_{l\in\N}$ be a regular measurable atlas, where
$U_l\cong B_l\times T_l$. Clearly, $i$ is measurable on the leaves
since the injectivity radius is a lower semicontinuous map.
Consider the Borel \sg-algebra associated to the compact-open
topology in $C(B_l,\R^{m^2})$. Then the Riemannian metric $g$ on
the chart $U_l$ can be considered as a measurable map $g:T \to
C(B_l,\R^{m^2})$, where $g(t)(x)$ is the matrix of coefficients of
$g$ at $(x,t)$ with respect to the canonical frame of $T\R^m$. In
fact, we can work with the closure $\overline{B_l}\times T$. Let
$\{\{U(p^n_i)\}_{i\in\N}\}_{n\in\N}$ be a sequence of open
coverings of $C(\overline{B_l},\R^{m^2})$, where $p^n_i\in
C(\overline{B_l},\R^{n^2})$, and $U(p^n_i)$ consists of the
functions $f\in C(\overline{B_l},\R^{m^2})$ such that $\|f -
p^n_{i}\|<2^{-n}$, using the norm of the maximum absolute value of
the coefficients in $\R^{m^2}$. Therefore
$T^n_i=\{g^{-1}(U(p^n_i))\}_{i\in\N}$ is a covering of $T$ by
measurable sets. By definition, for $t,t'\in T^n_i$,
$\|g(x,t)-g(x,t')\|<2^{-n}$ for all $x\in\overline{B_l}$.

Let $U_{l_1},\dots,U_{l_N}$ be a finite sequence of measurable
charts (it is possible that $U_{l_i}=U_{l_j}$ for some $i\neq j$)
and let $U=\bigcup_{j=1}^N U_{l_j}$. Then $U$ can be decomposed
into a countable family of product foliations $\FF_n$ such that
$\FF_n\cap U_{l_j}$ is saturated in $U_{l_j}$ for each $j$. Of
course, we can do the previous argument on each product foliation
of the given decomposition.

Since the family of finite collections of measurable charts is
countable, we have proved that the lamination $\FF$ is a countable
union of products $\{K^n_i\times T^n_i\}_{i\in\N}$, where each
$K^n_i$ is compact, $\|g(x,t)-g(x,t')\|<2^{-n}$ for all $x\in
K^n_i$ and $t,t'\in T^n_i$, and, for each $x\in\FF$, there exists
an expansive sequence $K^n_{i^n_x}\subset
K^{n+1}_{i^{n+1}_{x}}\subset\cdots$ meeting $x$ such that
$\bigcup_{n\in\N} K^n_{i^n_x}=L_x$, where $L_x$ is the leaf through
$x$. This final property is a consequence of the fact that these
$K^n_i$ are finite unions of closures of plaques in chains
of charts associated to each finite sequence of charts in $\UU$.

Finally, the injectivity radius map is measurable by the
lower semicontinuity relative to the variation of the metric. For
$n$ big enough, the Riemannian metric on the plaques of the
products $K^n_i\times T^n_i$ is so close to each other as we want.
Let $x\in X$ be a point where the injectivity radius is greater
than $r\in\R$. By the lower semicontinuity, choose $n$ such that
$x\simeq (x_0,t_0)\in \intr(K^n_i)\times T^n_i$ and the
injectivity radius of each point $(y,t)$, $(y,t)\in B_x\times
T^n_i$, is also greater than $r$; where $B_x\subset K^n_i$ is an
open neighborhood of $x_0$. Being the sequences of products
$\{K^n_i\times T^n_i\}_{n,i\in\N}$ a countable collection and
leaves second countable, we deduce that $i^{-1}(r,\infty]$ is
measurable for all $r\in\R$.
\end{proof}

\begin{defn}[Measurable triangulation \cite{Bermudez}]\label{d:measurabletriangulation}
Let $\bigtriangleup^n$ denote the canonical
$n$-simplex. A {\em measurable triangulation\/} is a family of triangulations on the leaves,
$\{\TT_L\}_{L\in\FF}$, which is {\em measurable\/} in the following sense:
  \begin{itemize}

    \item the sets of barycenters of $n$-simplices, $\BB^n$, are transversals of $\FF$; and

    \item the maps $\sg^n:\bigtriangleup^n\times\BB^n\to X$ are measurable, where $\sg^n(p,\cdot):\bigtriangleup^n\to L_p$ is the simplex of $\TT_{L_p}$ with barycenter $p$.

  \end{itemize}
A measurable triangulation is of
class $C^m$ if the functions $\sg^n_p$ are $C^m$.
\end{defn}

We work in this
section with $C^1$ measurable triangulations.

Let \TT\ be a measurable triangulation. Let $\TT^n$ denote the family
of $n$-faces of $\TT$ (the $n$-simplices of \TT\ without their boundaries).

\begin{prop}\label{p:tubeisolatedtransversal}
Let $T$ be an isolated transversal. There exists a measurable open
neighborhood $U(T)$ of $T$ such that the closures of its connected
components are disjoint and contain only one point of $T$. In
fact, $U(T)$ can be contracted to $T$ in a measurable way.
\end{prop}

\begin{proof}
Since $T$ is isolated and Borel, the function
$h:T\to\R\cup\{\infty\}$, defined by
$$
h(p)=\infim\{\,d_g(p,p')\mid p'\in(T\cap L_p)\setminus\{p\}\,\}\;,
$$
where $L_p$ denotes the leaf meeting $p$ and $d_g$
is the distance map on the leaves induced by the metric $g$, is measurable (by the measurability of $d_g$) and positive. The set
$h^{-1}(\infty)$ is measurable. Redefine $h$ in this set to be identically equal to $1$, and hence $0<h<\infty$. Now, let
$$
U(T)=\bigcup_{p\in T}B^g(p,\minim\{h(p),i(p)\}/2)\;,
$$
where $B^g(p,\varepsilon)$ is the $d_g$-ball in $L_p$
with center on $p$ and radius $\varepsilon$. This set is
measurable since $h$ and $i$ are measurable. Obviously, the
connected components are the balls
$B^g(p,\minim\{h(p),i(p)\}/{2})$ and satisfy the required
conditions. A measurable contraction to $T$ is given by the radial contraction on the tangent
space via the exponential map.
\end{proof}

\begin{defn}\label{d:comphomot}
Let $H:U\times \R\to X$ and $G:V\times \R\to X$ be tangential
deformations such that $H(U\times\{1\})\subset V$. Let
$H*G:U\times \R\to X$ be the tangential deformation defined by
$$H*G(x,t)=
\begin{cases}
H(x,2t) & \text{if $t\leq \frac{1}{2}$}\\
G(H(x,1),2t-1)& \text{if $\frac{1}{2}\leq t$.}
\end{cases}
$$
\end{defn}


\begin{lemma}\label{l:homotopytransversal}
Let $T$ be a standard Borel space, let $\R^m\times T$ be endowed
with the usual MT-structure and let $\pi:\R^m\times T\to T$ be the
canonical projection. Let $S$ be a transversal that meets each
plaque of $\R^n\times T$ at most in one point. Then there exists a
measurable homotopy $H:S\times \R\to \R^m\times T$ such that
$H(s,0)=s$ and $H(s,1)=(0,\pi(s))$.
\end{lemma}
\begin{proof}
Consider the measurable homotopy
$$
G:(\R^m\times T)\times \R\to\R^m\times T\;,\quad((v,t),s)\mapsto((1-s)v,t)\;.
$$
Then $H=G|_{S\times \R}$ satisfies the conditions of the statement.
\end{proof}

\begin{cor}\label{c:homotopy}
Let $T$ and $T'$ be transversals in a measurable chart $U$ which
are bijective via the canonical projection map. Then there exists
a measurable homotopy $H:T\times \R\to U$ such that $H(t,0)=t$ and
$H(t,1)\in T'\cap P_t$ for all $t\in T$, where $P_t$ is the plaque
containing $t$.
\end{cor}

\begin{defn}\label{d:chain}
A {\em chain of charts\/} of a measurable foliated atlas
$\UU=\{(U_i,\varphi_i)\}_{i\in\N}$ is a finite sequence,
$\CC=(U_{i_0},\dots,U_{i_n})$, such that $U_{i_j}\cap
U_{i_{j+1}}\neq\emptyset$ for all $j$. The chain of charts $\CC$ covers a path
$c:I=[0,1]\to X$ if there exists a partition
$0=t_0<t_1<\dots<t_n=1$ of $I$ such that $c([t_{j-1},t_j])\subset
U_{i_j}$ for all $j$. The {\em length\/} of a chain of charts $\CC=\{U_{i_0},\dots,U_{i_n}\}$ is $n$. In a similar way we can define a chain of plaques and a the length of a chain of plaques.
\end{defn}

Let $c:I\to X$ be a foliated path (i.e., a path contained
in one leaf). Any chain of charts $\CC=(U_{i_0},\dots,U_{i_n})$ covering $c$ induces a measurable holonomy map $h_\CC$ between transversals
containing $c(0)$ and $c(1)$ like in the
topological case.

\begin{lemma}\label{l:measurabletube}
Let $(X,\FF)$ be a measurable lamination that admits a regular
foliated measurable atlas. Let $c:I\to X$ be a foliated path. Let
$\CC=(U_{i_0},\dots,U_{i_n})$ be a chain of charts covering $c$, and
$h_\CC$ the measurable holonomy map induced by \CC. Let $T$ be the
domain of $h_\CC$. Then there exists a measurable homotopy
$H:T\times \R\to X$ such that $H(t,0)=t$ and $H(t,1)=h_\CC(t)$.
\end{lemma}

\begin{proof}
There exist $n-1$ transversals, $S_1,\dots,S_{n-1}$, such that
$S_k\subset U_{i_k}\cap U_{i_{k+1}}$, which meet only plaques that
cut these intersections, and only at one point (by
Theorem~~\ref{t:Srivastava}). By Corollary~~\ref{c:homotopy}, we
obtain the required homotopy.
\end{proof}

\begin{lemma}\label{l:holonomytransversals}
Suppose that \FF\ admits a regular foliated measurable atlas. Let
$h:T\to T'$ be a measurable holonomy map. Then there exists a
measurable homotopy $H:T\times \R\to X$ such that $H(t,0)=t$ and
$H(t,1)=h(t)$.
\end{lemma}

\begin{proof}
By Proposition \ref{p:Lusin1}, we can suppose that $T$ and $T'$ are
contained in transversals associated to charts in the regular
foliated measurable atlas. Observe that there exists a
countable number of chains of charts covering foliated paths
connecting points of $T$ and $T'$; the induced measurable holonomy
maps are denoted by $\{h_1,h_2,\dots\}$. The sets
$$
B_n=\{\,t\in T\mid h_{n}(t)=h(t)\,\}
$$
are measurable. By induction, define
\[
C_1=B_1\;,\quad C_n=B_n\setminus (C_1,\cup\dots\cup C_{n-1})\quad (n>1)\;.
\]
These transversals form a partition of $T$ and,
by Lemma~\ref{l:measurabletube}, there exist homotopies
$H^i:C_i\times \R\to X$ such that $H^i(t,0)=t$ and
$H^i(t,1)=h_{i}(t)=h(t)$. The Borel sets $H^i(C_i\times\{1\})$
form a partition of $T'$ since $h$ is a bijection. Combining these
homotopies, we obtain the desired homotopy.
\end{proof}


\begin{prop}\label{p:completetransversal}
Let $U$ be a categorical open set, and let $T$ be a complete
transversal. There exists a measurable contraction $H$ of $U$ so
that $H(U\times\{1\})\subset T$.
\end{prop}

\begin{proof}
Let $F$ be a contractible homotopy for $U$. Therefore
$T_F=F(U\times\{1\})$ is a measurable transversal. Since $T$ is a
complete transversal, by Proposition~\ref{p:Lusin1}, there exists a
countable partition of $T_F$ into measurable transversals
$\{T^i_F\}$ ($i\in\N$), and there are injective measurable holonomy maps
$h_i:T^i_F\to T$. By the above arguments, these holonomy maps induce a
measurable homotopy $G:T_F\times \R \to \FF$ such that $G(x,0)=x$
and $G(T_F\times\{1\})\subset T$. Then $H=F*G$ is the required
contraction.
\end{proof}

By Proposition~\ref{p:tubeisolatedtransversal}, $\TT^0$ is
contained in a categorical open set. Now, we
prove an analogous property for each $\TT^n$ for $0 < n \leq \dm\FF$.

\begin{prop}\label{p:TT'}
There exists a measurable triangulation $\TT'$ and
categorical open sets $U^n$ such that $\TT'^n\subset U^n$ for $0<
n \leq \dm\FF$.
\end{prop}
\begin{proof}
Let $e(x)$ be the $n$-face (n-simplex without boundary) containing
$x$. Using barycentrical division, we can suppose that all $n$-faces,
$0<n\leq\dm\FF$, are contained in an exponential ball centered at the
corresponding barycenter (a geodesic ball with radius smaller than the injectivity radius); this triangulation will be called $\TT'$. In fact, we can suppose that the diameter of $e(p)$ is smaller
than $i(p)/2$ for any barycenter $p$. Now, we construct a measurable open set
$U_n$ that contains $\TT'^{n}$ and such that each of its connected
components contains only one $n$-face and is contained in the
respective geodesic ball. This measurable open set contracts to
the set of barycenters of $\TT'^{n}$ by the exponential map, which completes the proof.

Let $\BB'^{n}$ denote the set of barycenters of $\TT'^n$. Let
$\rho:e(p)\to\R^+$ be a continuous function and let $N(e(p),\rho)$
denote the neighborhood of $e(p)$ consisting of the union of the balls of radius $\rho(x)$ in the geodesic orthogonal sections of $e(p)$ through $x$.
We define
$h^n:\TT'^n\to\R$ by $h(x)=d_g(x,\TT'^n\setminus e(x))$, which is
measurable since $g$ and $\TT'$ are measurable. Now, let
$\rho^n_p:e(p)\to\R^+$ be given by
$$\rho^n_p(x)=\frac{1}{2}\minim\{h(x),i(p)\}\;.$$ Clearly,
$U^n=\bigcup_{p\in\BB^n}N(e(p),\rho^n_p)$ is a measurable open set
that covers $\TT'^n$. Each open set $N(e(p),\rho^n_p)$ is contained
in the maximal exponential ball centered at $p$ by definition of
$\rho^n_p$ and the conditions satisfied by $\TT'$. These open sets are disjoint from each other. In fact, if $x\in N(e(p),\rho^n_p)\cap N(e(p'),\rho^n_{p'})$, then
$$
d_g(x,e(p))=d_g(x,\xi)
\leq \frac{1}{2}\,d_g(x,\TT'^n\setminus e(p)) \leq \frac{1}{2}d_g(x,e(p')) = \frac{1}{2}d_g(x,\xi')
$$
for certain $\xi\in e(p)$ and $\xi'\in e(p')$. By the symmetric argument,
$d_g(x,\xi')\leq\frac{1}{2}\,d_g(x,\xi)$, which is
contradiction. Therefore the exponential map defines a measurable
contraction of the measurable open set $U^n$ to $\BB'^n$.
\end{proof}

\begin{thm}[Dimensional upper bound]\label{t:dimensionalbound}
Let $T$ be a complete transversal for the $C^1$ measurable lamination $(X,\FF,\LB)$ with a $C^1$ measurable triangulation. Then $\Cat(\FF,\LB)\leq
(\dm\FF+1)\cdot\LB(T)$.
\end{thm}

\begin{proof}
Measurable laminations of class $C^1$ admit a $C^1$ measurable
triangulation and a leafwise Riemannian metric~\cite{Bermudez-Hector}. By the Proposition~\ref{p:TT'}, there
exists a categorical measurable open set $U^n$
containing each set $\TT^n$ associated to a measurable
triangulation for $0\leq n \leq \dm\FF$. Hence the sets $U^n$
cover $X$. By Proposition \ref{p:completetransversal},
$\tau_\LB(U^n)\leq \LB(T)$ for $0\leq n \leq \dm \FF$.
\end{proof}

This theorem has important consequences.

\begin{cor}
Let $(X,\FF)$ a minimal $C^1$ lamination. Let \LB\ be a regular transverse
invariant measure of $\FF$ without atoms. Then $\Cat(\OO(\FF),\LB)=0$.
\end{cor}

Recall that a transverse invariant measure of a foliated measurable space is called {\em ergodic\/} if it is finite in a complete transversal and any
saturated measurable set has null or full measure.

\begin{cor}\label{c:ergodic}
Let $(X,\FF,\LB)$ be a $C^1$ measurable lamination with an ergodic transverse invariant measure without atoms. Then $\Cat(\FF,\LB)=0$.
\end{cor}

\begin{proof}
Of course, in a ergodic lamination without atoms there exists complete transversals with arbitrarily small measure.
\end{proof}

\section{Cohomological lower bound}

In this section, we give a version of the useful cohomological lower bound of the classical LS category, stating that $\Nil(H^*(M,\Gamma))\le\Cat(M)$ for any manifold $M$, where $\Nil(H^*(M,\Gamma))$ denotes the nilpotence order of the cohomology ring $H^*(M,\Gamma)$ with coefficients in any ring $\Gamma$ \cite{Dubrovin-Novikov-Fomenko}. For this purpose, we give an idea of the cohomology of MT-spaces \cite{Bermudez,Heitsch-Lazarov,Menino2}.

We suppose that $\Gamma$ is a standard ring; i.e., $\Gamma$ is a standard space and a ring where all the operations are measurable.

\begin{defn}[Measurable prism \cite{Bermudez,Bermudez-Hector}]
A {\em meaurable prism\/} is a product of a standard Borel space
$T$ and a linear region of $\R^N$ (for instance a polygon) with
the standard MT-structure. A {\em measurable simplex\/} is a
measurable prism where the topological fiber is a canonical
$n$-simplex $\bigtriangleup^n$. A {\em measurable singular
simplex\/} on $X$ is an MT-map $\sg:\bigtriangleup^n\times T \to X$.
\end{defn}

Let  $\omega$ be a usual {\em singular $n$-cochain\/} over a
coefficient ring $\Gamma$. It is said that $\omega$ is {\em
measurable\/} if $\omega_{\sg}:T\to\Gamma$, $t\mapsto
\omega(\sg_{|\bigtriangleup\times \{t\}})$, is measurable for all
measurable singular $n$-simplex $\sg$. The set of measurable
cochains form a subcomplex of the complex of usual cochains since
the coboundary operator $\delta$ preserves measurability. This is called the
{\em measurable subcomplex\/} and denoted by $C^*_{MT}(X,\Gamma)$, and the corresponding restriction of $\delta$ is denoted by $\widetilde{\delta}$. The {\em measurable singular cohomology\/} is defined as usual by $H^n_{\MT}(X,\Gamma)=\Ker \widetilde{\delta}_n/\Image\widetilde{\delta}_{n-1}$.

The usual cup product induces a product on
measurable cochains, which of course preserves measurability since the operations in $\Gamma$ are measurable. The usual formula
$\widetilde{\delta}(\omega\smile\theta)=\widetilde{\delta}\omega\smile\theta
+ (-1)^n\omega\smile \widetilde{\delta}\theta$ holds. Therefore it
induces a cup product in measurable cohomology, obtaining the
graded ring $(\bigoplus_{n\geq 1}H^n_{\MT}(X,\Gamma), +, \smile)$.

Let $f:X\to Y$ be an MT-map. Clearly, $f$ induces a cochain map $f^*:C^*_{\MT}(Y,\Gamma)\to C^*_{\MT}(X,\Gamma)$
defined by $f^*(\omega)(\sg)=\omega(f\circ\sg)$. This cochain map
commutes with $\widetilde{\delta}$, and therefore it induces a
homomorphism between the corresponding measurable cohomology groups.

Let $U\subset X$ be an MT-subspace of $X$. The inclusion map
determines a chain map $i^*:C^*_{\MT}(X,\Gamma)\to
C^*_{\MT}(U,\Gamma)$. The cochain complex determined by $\Ker(i^*)$
will be denoted by $C^*_{MT}(X,U,\Gamma)$. The cochains in this
cochain complex are usual cochains vanishing on the singular
simples contained in $U$. The cohomology groups associated to this
chain complex will be called the {\em measurable relative
cohomology\/} groups of $(X,U)$. By using the Ker-CoKer Lemma,
there exists a long exact sequence of cohomology groups like in
the classical case (the details are easy to check),
$$
\cdots\to H^{n}_{\MT}(X,U,\Gamma)\to H^n_{\MT}(X,\Gamma)\to
H^n_{\MT}(U,\Gamma)\to H^{n+1}_{\MT}(X,U,\Gamma)\to\cdots
$$

\begin{prop}[Invariance by measurable tangential homotopy]\label{p:invariancetransversals}
Let $f,g:X\to Y$ be MT-homotopic maps. Then $f^*$ and $g^*$
induce the same homomorphism in measurable singular cohomology.
\end{prop}

\begin{proof}
The proof is a trivial consequence of the classical proof for
singular cohomology. The measurable homotopy induces a chain
homotopy between $f_*$ and $g_*$ at the level of the chain
complex. The definition is given by cutting the space
$\bigtriangleup^n\times [0,1]$ into a finite number of
$n+1$-prisms $\Pi_i:\bigtriangleup^{n+1}\to\bigtriangleup^n\times [0,1]$. Let $H$ be the measurable homotopy between $f$ and
$g$. The prism operator $P:C^{n+1}(X,\Gamma)\to C^{n}(X,\Gamma)$, defined by
$P(\omega)(\sg)=\sum_i I_i\,\omega(H\circ\sg_{|\Pi_i})$ \cite{Hatcher}, where $I_i$
is the orientation factor, is a cochain homotopy between $f^*$ and
$g^*$ preserving measurability. Hence $f^*$ and $g^*$ induce
the same homomorphism in measurable cohomology.
\end{proof}

Let $C_{\MT}^*(U+V)$ be the cochain complex given by the measurable
cochains which vanish on measurable singular simplices that
do not lie in either $U$ or $V$.
Now, $H_{\MT}^*(\FF_{U\cup V})$ is isomorphic to $H_{\MT}^*(U+V)$ via the restriction
map \cite{Menino2} like in the usual singular cohomology. In fact, by using the 5-Lemma, $H_{\MT}^*(\FF,\FF_U)$ is isomorphic to $H_{\MT}^*(\FF,U+V)$.

An easy computation shows that $H^n_{\MT}(X,\Gamma)=0$ for $n\geq 1$ when $X$ is an MT-space with the discrete topology. That is the case of any measurable transversal of a measurable lamination.

\begin{cor}
The cup product induces a cup product in measurable relative
cohomology:
$$
\smile\;:H^n_{\MT}(\FF,\FF_U,\Gamma)\times H^m_{\MT}(\FF,\FF_V,\Gamma)\to
H^{n+m}_{\MT}(\FF,\FF_{U\cup V},\Gamma)\;.
$$
\end{cor}

\begin{cor}\label{c:nilpotencebound}
$\Nil(\bigoplus_{n\geq 1} H^n_{\MT}(\FF,\Gamma),+,\smile)\leq\Cat \FF$.
\end{cor}

\begin{proof}
Let $\{U_1,\dots,U_N\}$ be a covering by tangentially contractible measurable open sets. The map $H^*_{\MT}(\FF,\FF_{U_i},\Gamma)\to H^*_{\MT}(\FF,\Gamma)$ of the cohomological exact sequence is onto since each $U_i$ is categorical. Let $x_1,\dots,x_N$ be cohomology clases in $H^*_{\MT}(\FF)$, and take a preimage $\omega_i$ of each $x_i$ in $H^*_{\MT}(\FF,\FF_{U_i},\Gamma)$. Therefore $\omega_1\smile\dots\smile \omega_N\in H^*_{\MT}(\FF,\FF,\Gamma)=0$, and this product projects to $x_1\smile\dots\smile x_N$ in $H^*_{\MT}(\FF,\Gamma)$.
\end{proof}

Of course, singular measurable cohomology is a complicated object. Fortunately, we can work with simplicial measurable cohomology \cite{Bermudez,Heitsch-Lazarov,Menino2}.

Let $\TT$ be a measurable triangulation. An {\em $n$-cochain\/} over a measurable ring
$\Gamma$ is a measurable map $\omega:\BB^n\to \Gamma$, where the
barycenters in $\BB^n$ are identified to the corresponding $n$-simplices. Let $C^n(\TT,\Gamma)$ denote the set of
simplicial n-cochains, which is endowed with a ring structure induced by
$\Gamma$. Define the coboundary operator $\delta:C^n(\TT,\Gamma)\to
C^{n+1}(\TT,\Gamma)$ as usual: for $\omega:\BB_{n+1}\to\Gamma$, let $\delta\omega:\BB_{n+1}\to\Gamma$ be given by
$\delta\omega(b)=\sum_{\bigtriangleup^n_p\subset\partial\bigtriangleup^{n+1}_b}(-1)^n\,\omega(p)$,
where $\bigtriangleup^k_x$ denotes the $k$-simplex with barycenter $x$ and the simplices of the boundary are chosen in the usual order~\cite{Hatcher}. Clearly, $\delta^2=0$, and we can define the cohomology groups as usual:
$H^n(\TT,\Gamma)=\Ker d_n/\Image d_{n-1}$.

\begin{prop}[\cite{Menino2}]\label{p:equivsimplicialsingular}
Let $(X,\FF)$ be a measurable lamination that admits a
measurable triangulation. Then the measurable singular cohomology
groups are isomorphic to the measurable simplicial ones.
\end{prop}

\begin{cor}
The measurable simplical cohomology does not depend on the measurable triangulation.
\end{cor}

\begin{rem}
Of course, we can define the concept of measurable CW-structure in a similar way to a measurable triangulation. This notion gives a measurable CW-complex and a measurable cellular cohomology. It can be proved that it is isomorphic to the measurable singular cohomology with arguments similar to the above ones.
\end{rem}

\begin{exmp}\label{e:KroneckerCohom}
Let $(T^2,\FF_\alpha)$ be the K\"onecker flow, considered as a
suspension of the rotation $R_{\alpha}:S^1\to S^1$ of $2\pi\alpha$
radians. The case where $\alpha$ is rational is trivial (it is a foliation with compact leaves). Then suppose that $\alpha$ is irrational. The projection of
$[0,1]\times S^1$ to $T^2$, given by the suspension of $R_\alpha$, induces a
measurable triangulation of $\FF_\alpha$, where the $0$-skeleton is
the projection of $\{0\}\times S^1$ and the $1$-skeleton is the
projection of $[0,1]\times S^1$; the set of
barycenters is the projection of $\{1/2\}\times S^1$. Of course,
measurable cochains of degrees $0$ and $1$ are measurable maps $f:S^1\to
\Z_2$. In \cite{Menino2}, it is proved that the $1$-cochain $\omega=1:S^1\to \Z_2$
represents a non-zero element in $H^1_{\MT}(\FF_\alpha,\Z_2)$, showing that
$H^1(\FF_\alpha,\Z_2)\ne0$. By Corollary~\ref{c:nilpotencebound} and the dimensional bound, it follows that $\Cat(\FF_\alpha)=2$.

In higher dimension, let $\FF_{\alpha_1,\dots,\alpha_n}$ be the foliation on $T^{n+1}$ given by the suspension of the rotations $R_{\alpha_1},\dots,R_{\alpha_n}$ of $S^1$, where $\alpha_1,\dots,\alpha_n$ are $\Q$-linear independent. Each
leaf of $\FF_{\alpha_1,\dots,\alpha_n}$ is a hyperplane dense in $T^{n+1}$. Let $[0,1]^n\subset\R^n$ be the unit cube and let $T\equiv S^1$.
The product $[0,1]^n\times T$ gives a {\em measurable CW-structure\/} on
$\FF$ given by the projection $p:[0,1]^n\times
T\to\FF$ defined by the suspension. Let $\omega=1:T\to\Z_2$, which is
a CW-cochain of dimension $n$. Let $\tau_i$ the CW-cochain of
dimension $1$ satisfying $\tau_i(p_{|[0,1]_j\times T})=\delta_{ij}$, where
$$
[0,1]_j=\{\,(x_1,\dots,x_n)\in[0,1]^n\mid x_i=0\ \text{for}\ i\neq j\,\}\;.
$$
Clearly $\omega=\tau_1\smile\dots\smile \tau_{n}$. Moreover $\omega$ represents a non-zero element in $H^n_{\MT}(\FF_{\alpha_1,\dots,\alpha_n},\Z_2)$ \cite{Menino2}, obtaining $H^n_{\MT}(\FF_{\alpha_1,\dots,\alpha_n},\Z_2)\ne0$. Then Corollary~\ref{c:nilpotencebound} and the dimensional upper bound yields $\Cat(\FF_{\alpha_1,\dots,\alpha_n})=n+1$.
\end{exmp}

\begin{rem}
We do not discuss here about the possibility of a similar lower bound for the \LB-category. In many general situations, this invariant is zero and a lower bound is not so interesting. In the cases where there exists a complete transversal $T$ meeting each leaf at exactly one point, the number $\int_T \Nil H^*(L_t,\Gamma)\,d\LB(t)$ is well defined and gives a lower bound for the \LB-category.
\end{rem}

\section{Critical points}\label{s:critical measurable}

As suggested by the classical theory, because of the possible applications, it is interesting to consider Hilbert laminations to study the relation between our measurable versions of the tangential category and the critical points of a smooth function.

Recall the following terminology. A {\em Hilbert manifold\/} is a separable, Hausdorff space endowed with an atlas where the charts are homeomorphisms to open subsets of separable Hilbert spaces. A $C^2$ function on a complete $C^2$ Hilbert manifold, $f:M\to\R$, is called {\em Palais-Smale\/} whenever, for any sequence $(p_n)$ in $M$, if $(f(p_n))$ is bounded and $(df(p_n))$ converges to
zero, then $(p_n)$ contains a convergent subsequence. If $M$ is
compact (in the finite dimensional case), then all differentiable maps are Palais-Smale. Let
$\Crit(f)$ be the set of critical points of $f$. In this section,
we adapt a theorem due to J.~Schwartz \cite{Schwartz}, which states that, for a
bounded from below Palais-Smale function on a $C^2$ Hilbert manifold, $f:M\to \R$, we have
$\Cat(M)\leq\#\Crit(f)$.

For the rest of the section, we consider $C^2$ measurable Hilbert
laminations. Their definition is analogous to the definition of
measurable lamination: the leafwise topological model is now given by open
balls in a separable Hilbert space instead of $\R^n$, and the
tangential part of the changes of coordinates are $C^2$ maps
between open subsets of Hilbert spaces. Each leaf is a Hilbert
manifold, which is endowed with a Riemannian metric that varies in
a measurable way in the ambient space. Observe that the main
difference with the finite dimensional case is that the leaves may
not be locally compact. We also suppose that the leafwise
Riemannian metric is complete.

To define the \LB-category of a measurable Hilbert lamination with a transverse invariant measure $\Lambda$, we consider only contractible measurable open sets. The reason is that Proposition~\ref{p:posit} seems to be difficult to generalize to this infinite dimensional setting, as well as other details about measurability. However, Proposition~\ref{p:coherentextension} holds for measurable Hilbert laminations too, and therefore we can use the coherent extension $\widetilde{\Lambda}$ of $\Lambda$.

Notice that the differential map of a function varies in a measurable way in
the ambient space, since its definition is a limit of measurable maps. The set
of $C^2$ MT-functions $\FF\to\R$ will be noted by $C^{2}(\FF)$. For $f\in C^{2}(\FF)$, let $\Crit_{\FF}(f)=\bigcup_{L\in
\FF}\Crit(f|_L)$.

The cotangent bundle $T\FF^*$ is a measurable vector bundle, whose zero section $\theta:X\to T\FF^{*}$ is measurable with measurable image. For any $f\in C^{2}(\FF)$, its differential map $df:X\to T\FF^{*}$ is measurable, and we have $\Crit_{\FF}(f)=df^{-1}(\theta(X))$. Thus $\Crit_{\FF}(f)$ is measurable.

Recall that a $C^1$ {\em isotopy\/} on a $C^1$-Hilbert manifold $M$ is a
differentiable map $\phi:M\times\R\to M$ such that $\phi_t=\phi(\cdot,t):M\to M$
is a diffeomorphism $\forall t\in [0,1]$ and $\phi_0=\id_M$.

\begin{defn}[Measurable tangential isotopy]
Let $(X,\FF)$ be a $C^{1}$ measurable Hilbert lamination. A {\em
measurable tangential isotopy\/} on $(X,\FF)$ is a $C^{1}$ map
$\phi:X\times \R\to X$ such that the functions $\phi_t:X\to X$ are
MT-diffeomorphisms $\forall t$, with $\phi_0=\id_X$, and the map $\phi:\FF\times\R\to\FF$, with $\phi(x,t)=\phi_t(x)$, is $C^1$, where $\FF\times\R$ is the $C^1$ measurable lamination in $X\times\R$ with leaves of the form $L\times\R$ for $L\in\FF$. In particular, the maps $\phi^x:\R\to X$, with $\phi^x(t)=\phi_t(x)$, are differentiable $\forall x\in X$.
\end{defn}

\begin{rem}\label{r:openinv}
Let $\phi$ be a measurable tangential isotopy on $X$ and let
$U\subset X$ be a measurable open set. Then
$\Cat(U,\FF,\LB)=\Cat(\phi_t(U),\FF,\LB)$ and $\Cat(U,\FF)=\Cat(\phi_t(U),\FF)$ for all $t\in\R$.
\end{rem}

\begin{exmp}[Construction of a measurable tangential isotopy \cite{Palais}]
\label{e:measurable tangential isotopy}
A tangential isotopy can be constructed on a Hilbert manifold by using a $C^1$ tangent vector field $V$. There exists a flow $\phi_t(p)$ such that $\phi_0(p)=p$, $\phi_{t+s}(p)=\phi_t(\phi_s(p))$ and $d\phi_t(p)/dt=V(\phi_t(p))$. From the way of obtaining $\phi$ \cite{Palais,Crandall-Pazy}, it follows that the same kind of construction for a measurable $C^1$ tangent vector field on a measurable Hilbert lamination $(X,\FF)$ induces a measurable isotopy on $(X,\FF)$.

Now, we obtain a measurable isotopy from the gradient flow. It will be modified by a control function $\alpha$ in order to have some control in the deformations induced by the corresponding isotopy.
Let $\nabla f$ be the gradient tangent vector field of $f$; i.e., the
unique tangent vector field satisfying $df(v)=\langle v,\nabla
f\rangle$ for all $v\in T\FF$. Take the $C^{1}$ vector field $V=-\alpha(|\nabla f|)\,\nabla f$, where $\alpha:[0,\infty)\to\R^+$ is $C^\infty$, $\alpha(t)\equiv 1$ for $0\leq t\leq 1$, $t^2\alpha(t)$ is monotone non-decreasing and $t^2\alpha(t)=2$ for $t\geq 2$. The flow $\phi_t(p)$ of $V$ is defined for $-\infty<t<\infty$ \cite{Schwartz}, and it is called the {\em modified gradient flow\/}.
\end{exmp}

Let us define a partial order relation ``$\ll$'' for the critical
points of $f$. First, we say that $x<y$ if there exists a regular
point $p$ such that $x\in\alpha(p)$ and $y\in\omega(p)$, where
$\alpha(p)$ and $\omega(p)$ are the $\alpha$- and $\omega$-limits of
$p$. Then $x\ll y$ if there exists a finite sequence of critical
points, $x_1,\dots,x_n$, such that $x < x_1 < \dots< x_n < y$.

Let $\gamma(x)$ denote
the $\phi$-orbit of each
point $x$.


\begin{lemma}\label{l:lowaproxim}
Let $T\subset X$ be an isolated transversal. Then there exists a measurable
open set $U$ containing $T$ such that
$\Cat(U,\FF,\LB)\leq \LB(T)$. We can suppose also that $U$ is tengentially categorical contracting to $T$.
\end{lemma}

\begin{proof}
A tubular neighborhood of $T$ contracts to $T$ (see Lemma~\ref{l:measurabletube} and observe that its proof generalizes to measurable Hilbert laminations). Hence its relative category is less or equal than $\LB(T)$.
\end{proof}

In the following proposition, $\{W_1,W_2,\dots\}$ denotes a
foliated measurable atlas. Let $\Crit_\FF^\infty(f)$ be the union
of plaques that contain infinite critical points of $f$; observe
that this is a measurable open set.

\begin{prop}
If $\widetilde{\LB}(\Crit_\FF^\infty(f))>0$, then
$\widetilde{\LB}(\Crit_\FF(f))=\infty$.
\end{prop}
\begin{proof}
For each chart $W_i$, let $\pi_i:W_i\to T_i$ be
the transverse projection. Since $\widetilde{\LB}(\Crit_\FF^\infty(f))>0$, we
have $\LB(\pi_i(\Crit_\FF^{\infty}(f)\cap W_i)>0$ for some $i\in\N$.
Therefore 
\begin{multline*}
\widetilde{\LB}(\Crit_\FF^\infty(f))\geq
\int_{\pi_i(\Crit_\FF^\infty(f)\cap W_i)}\#(\Crit_\FF(f)\cap
\pi_i^{-1}(t))\,d\LB(t)\\
=\infty\cdot\LB(\pi_i(\Crit_\FF^\infty(f)\cap W_i)) =\infty\;. \qed
\end{multline*}
\renewcommand{\qed}{}
\end{proof}

\begin{rem}\label{r:critinf}
The set $\Crit_\FF^\infty(f)$ contains all non-isolated critical
points of $\Crit_\FF(f)$. If $\widetilde{\LB}(\Crit(f))<\infty$,
then the saturation of $\Crit_\FF^\infty(f)$ is a null-transverse
set. Hence we can restrict our study to the case where all
critical points are isolated.
\end{rem}

The definition of a Palais-Smale condition is needed for a version of the Lusternik-Schnirelmann Theorem on Hilbert manifolds. For measurable Hilbert laminations, it could be adapted by taking functions that satisfy the Palais-Smale condition on all (or almost all) leaves. But this is very restrictive because it would mean that the set of relative minima meets each leaf in a relatively compact set (which is non-empty when $f$ is bounded from below), and therefore there would exist a complete transversal meeting each leaf at one point. Thus, instead, we use the following weaker condition.

\begin{defn}\label{d:PSmeasurable}
A measurable {\em $\omega$-Palais-Smale\/} (or {\em $\omega$-PS\/}) function is a function $f\in
C^{2}(\FF)$ such that any $\phi$-orbit have non empty $\omega$-limit and , for any $p\in\Crit_\FF(f)$, the set $\{\,x\in\Crit_\FF(f)\mid p\ll x\,\}$ is compact, and this set is empty if and only if $p$ is a relative minimum. A measurable {\em $\alpha$-Palais-Smale\/} (or {\em $\alpha$-PS\/}) function is defined analogously by considering the set $\{\,x\in\Crit_\FF(f)\mid x\ll p\,\}$.
\end{defn}

Of course, $f$ is $\omega$-PS if and only if $-f$ is $\alpha$-PS.

\begin{lemma}\label{pr:measurableflow}
Suppose that $\Crit_\FF(f)$ is an isolated transversal. The modified gradient flow $\phi$ \upn{(}see Example~\ref{e:measurable tangential isotopy}\upn{)}
satisfies the following properties:
\begin{itemize}
  \item[(i)] The flow runs towards lower level sets of $f$, i.e., $f(p)\geq f(\phi_t(p))$ for $t>0$.

  \item[(ii)]The invariant points of the flow are just the critical points of $f$.

  \item[(iii)] A point is critical if and only if $f(\phi_t(p))=f(p)$ for some $t\neq 0$.

  \item[(iv)] The points in the $\alpha$- and $\omega$-limits are critical points if they are non empty.

\end{itemize}
\end{lemma}

\begin{proof}
These properties can be proved in each leaf, considered as a $C^2$
Hilbert manifold, where (i),~(ii) and~(iii) follow from the work of J.~Schwartz \cite{Schwartz}.

Under these conditions, the $\alpha$- and $\omega$-limits are
connected sets that consist of critical points if they are non-empty (by using (i),~(ii) and~(iii)). If $\omega(p)$ is infinite, then all of its points
are non-isolated, contradicting the assumption.
\end{proof}

\begin{defn}[Critical sets]\label{d:partcrit}
For a measurable $\omega$-PS function, the set of minima is non-empty in any leaf. Let $p$ be a critical point. By the properties of the flow $\phi$, either $p$ is a relative minimum, or there
exists another critical point $x$ such that $p<x$. Define $M_0,M_1,\dots$ inductively by
\begin{align*}
M_0&=\{\,x\in \Crit_\FF(f)\mid\not\exists\ y\ \text{such that}\ x\ll y\,\}\;,\\
M_i&=\{\,x\in\Crit_\FF(f)\mid\forall y\,\ x\ll y\ \Rightarrow\ y\in M_0\cup\dots\cup M_{i-1}\,\}\;.
\end{align*}
Clearly, $M_0$ contains all relative minima on the leaves. We also
set $C_0(f)=M_0$ and $C_i(f)=M_i\setminus(M_0\cup\dots\cup
M_{i-1})$. Observe that, if $x\in\omega(p)$ and $x\in C_i(f)$ for
some $i$, then $\omega(p)\subset C_i(f)$. There is an analogous
property for the $\alpha$-limit. The notation $C_i$ will be used
if there is no confusion. Let $p\ll p^*$. Then $i_{p^*}< i_p$,
where $i_{p}$ and $i_{p^*}$ are the indexes such that $p\in
C_{i_p}$ and $p^*\in C_{i_{p^*}}$.
\end{defn}

The set of relative minima of a bounded from below measurable $\omega$-PS function is always non-empty in any leaf.

\begin{thm}\label{t:MeasTanCrit}
Let $(X,\FF)$ be a measurable Hilbert lamination endowed with a
measurable Riemannian metric on the leaves, and let $f$ be a measurable $\omega$-PS function on $X$. Suppose that $\Crit_\FF(f)$ is an isolated transversal. Then
$\Cat(\FF)\leq \#\{\text{critical sets of}\ f\}\;.$
\end{thm}

\begin{thm}\label{t:LSMedFol}
Let $(X,\FF,\LB)$ be a measurable Hilbert lamination endowed with a
measurable Riemannian metric on the leaves and with a transverse
invariant measure, and let $f$ be a measurable $\omega$-PS function on $X$. Then
$\Cat(\FF,\LB)\leq\widetilde{\LB}(\Crit_\FF(f))$.
\end{thm}

\begin{rem}
Notice also that Theorem~\ref{t:MeasTanCrit} gives a slight sharpening of the classical theorem of Lusternik-Schnirelmann, since the number of critical sets may be finite even when the number of all critical points are infinite. We see this in the following example. Let $f:\R\to\R$, $f(x)=\sin(x)$, which is not an Palais-Smale function in the classical sense, but it satisfies our measurable Palais-Smale condition in the one leaf lamination $\R$. An easy computation shows that there are two critical sets, the set of relative minima and the set of relative maxima, yielding the inequality $\Cat(\R)\leq 2$. On more complicated examples, this improved version could be used to find better upper bounds of the classical LS category.
\end{rem}

\begin{rem}\label{r:medtube}
From the existence of a measurable Riemannian metric, two disjoint
isolated transversals, can be separated by measurable
open sets. In fact, we can suppose that the closures of the connected
components of these measurable open sets contain only one point of
these transversals (see
Proposition~\ref{p:tubeisolatedtransversal}).
\end{rem}

\begin{lemma}\label{l:measurableconvergency}
Let $(X,\FF)$ be a measurable Hilbert lamination and let
$f_n:(X,\FF)\to(X,\FF)$ be a sequence of MT-maps. Suppose that
$(f_n(x))$ converges for all $x\in X$. Then $\lim_n f_n$ is measurable.
\end{lemma}

\begin{proof}
Measurable open sets generate the \sg-algebra of
$(X,\FF)$, in fact, by Theorem~\ref{t:srivastava} the measurable foliated charts are a generating set also. Then it is enough to prove that $(\lim_n f_n)^{-1}(V)$ is
measurable for any foliated chart $V$. For each $V\equiv B\times T$ there exists a sequence of measurable closed sets $\{F_n\}_{n\in\N}$ such that $V=\bigcup_n F_n$. For instance, if $B$ is an open ball we can take $F_n\equiv \overline{B_n}\times T$, where $B_n$ are open balls of smaller radius than $B$ but converging to it. Now, it is clear that
\begin{align*}
  (\lim_n f_n)^{-1}(V)&
  =\{\,x\in X\mid\exists N\ \text{such that}\ f_n(x)\in F_N\ \forall n\geq
N\,\}\\
&=\bigcup_{N=1}^\infty\bigcap_{n\geq N}f^{-1}_n(F_N)\;,
\end{align*}
 which is, clearly, a measurable set.
\end{proof}

\begin{rem}
It is known, in basic measure theory, that the limit of real measurable functions is also measurable. Lemma~\ref{l:measurableconvergency} is not a direct consequence from this fact because the measurable structure is not the Borel \sg-algebra corresponding to the topology.
\end{rem}

\begin{cor}\label{c:partition}
The family $\{C_i\}$ is a measurable partition of $\Crit_\FF(f)$.
\end{cor}

\begin{proof}
Clearly, this family is a partition by
the properties of $\phi$. To show that each $C_i$ is
measurable, observe that $\Crit_\FF(f)$ is a transversal, and
$X\setminus \Crit_\FF(f)$ is measurable with open intersection with each leaf. The
 $\alpha$- and $\omega$-limit functions, $\alpha,\omega:X\to\Crit_\FF(f)$, are measurable by
Lemma~\ref{l:measurableconvergency} since $\alpha=\lim_n\phi_{-n}$
and $\omega=\lim_n\phi_n$. Observe that
$\alpha(X\setminus\Crit_\FF(f))=\Crit_\FF(f)\setminus C_0$.
Therefore $C_0$ is measurable. By Remark~\ref{r:medtube}, there
exists a measurable open set $U_0$ containing $C_0$ and separating
it from $\Crit_\FF(f)\setminus C_0$; in fact, each connected
component contains only one point of $C_0$. Take the measurable
open set $O_0=\bigcup_{n\in\N}\phi_{-n}(U_0)$. It is easy to
see that
$$\alpha(X\setminus(O_0\cup\Crit_\FF(f)))=\Crit_\FF(f)\setminus
(C_0\cup C_1)\;.$$ Therefore $C_1$ is a measurable set. By a
recursive argument, we obtain that all sets $C_i$ are measurable.
\end{proof}

\begin{lemma}\label{l:measurable tubular neighborhoods}
Let $(X,\FF)$ be a measurable Hilbert lamination and let $(Y,\GG)$ be a finite dimensional lamination such that $Y\subset X$ is a measurable open  set and the inclusion map is an MT-embedding considering in $Y$. Then there exists a countable family of measurable foliated charts of $(X,\FF)$, $\{U_n\equiv B_n\times T_n\}_{n\in\N}$, covering $Y$ and such that each fiber $B_n\times\{t\}$ is a foliated chart (in a topological sense) of $\GG$ in a leaf of $\FF_Y$.
\end{lemma}
\begin{proof}
Let $U$ be a measurable foliated chart of \FF. Since $Y$ is open in $X$, $U\cap Y$ is also measurable and open. Hence a measurable foliated atlas induces a covering of $Y$ by measurable open sets and a measurable foliated atlas of $(Y,\FF_Y)$. Since $\GG$ is a sublamination of $\FF_Y$, by Theorem~\ref{t:Srivastava}, we can choose a measurable foliated atlas such that the plaques are products of the form $B^k\times B'$ where $B^k$ are open balls in $\R^k$ where $\dim \GG=k$ and $B'$ is a ball in the orthogonal complement of $\GG$ in a leaf of $\FF$ centered in the origin. Of course the projection $\pi:(B^k\times B)\times T\to (B^k\times\{0\})\times T$ defines an MT-map that works like a tubular neighborhood of a plaque of a family of leaves of $\GG$. Of course $B^k\times (B'\times T)$ is a measurable chart for $\GG$, hence the family of tubular neighborhoods is a measurable atlas of $\GG$ and $\FF_Y$ simultaneously.
\end{proof}

\begin{proof}[Proof of Theorem \ref{t:MeasTanCrit}]
By Remark~\ref{r:medtube} and Lemma~\ref{l:lowaproxim}, there
exists a disjoint family of measurable open sets,
$\{U_i\}$ ($i\in\N\cup\{0\}$), such that each $U_i$ contains and contracts to $C_i$, and $\Cat(U_i,\FF,\LB)\leq\LB(C_i)$. We also assume that $\overline{U_i}\subset \widetilde{U_i}$, where $\widetilde{U_i}$ is another categorical measurable open set that contracts to $C_i$.

Let $U'_0=\bigcup_{n\in\N}\phi_{-n}(U_0)$. This set is open since $C_0(f)$ consists of relative minima, and contracts to $C_0(f)$ by using the MT isotopy $\phi$. The set $X_1=X\setminus U'_0$ is a measurable closed set, and $C_1(f)$ is the set of relative minima of the restriction $f|_{X_1}$. The set $X_1$ consists of the critical points that do not belong to $C_0(f)$ and the regular points connecting these critical points according to the relation ``$\ll$''. Let $F_1=U_1\cap X_1$ and $F'_1=\bigcup_{n\in\N}\phi_{-n}(F_1)$. The set $F'_1$ is open in $X_1$ and closed in $X$. Let us prove that $F'_1$ is contained in a measurable open set $U'_1$ such that there exists a measurable deformation $H$ with $H(U'_1\times\{1\})\subset \widetilde{U_1}$.

Of course, $(X\setminus\Crit_\FF(f),\phi)$ is a measurable lamination of dimension $1$, where the leaves are the flow lines of $\phi$. These flow lines are embedded submanifolds and they admit a countable covering by measurable foliated charts in the sense of Lemma~\ref{l:measurable tubular neighborhoods}. Therefore there exists a measurable atlas of $(X\setminus\Crit_\FF(f),\phi)$, $\{(W_n,\varphi_n)\}_{n\in\N}$, with $\varphi_n:W_n\to B^1\times B_n\times T_n$, where $B^1$ is an open interval in $\R$,  $B_n$ is an open ball centered at the origin in a separable Hilbert space, and $T_n$ is a standard space. We can suppose also that this measurable atlas is locally finite and let $\pi_n:W_n\to \varphi_n^{-1}(B^1\times\{0\}\times T_n)$ be given by the canonical projection $B^1\times B_n\times T_n\to B^1\times\{0\}\times T_n$.

By similar arguments, we can suppose that $F'_1\subset\bigcup_nW_n\subset \bigcup_{i\in\N}\phi_{-i}(U_1)$, $\varphi_n(B^1\times \{0\}\times T_n)\subset F'_1$ and $\bigcup_nW_n$ is a {\em semisaturated\/} set; i.e., if $x\in\bigcup_nW_n$ then $\phi_t(x)\in\bigcup_nW_n$ for all $t\in [0,\infty)$. Let $\{\lambda_n\}$ be a measurable partition of unity subordinated to $\{W_n\}$ \cite{Bermudez-Hector} such that each $\lambda_n$ is continuous on $\varphi_n^{-1}(B^1\times B_n\times\{z\})$ for all $z\in T_n$. For each $x\in \bigcup_nW_n$, let $I(x)\subset\N$ be the set of numbers $n$ such that the semiorbit $\phi_{[0,\infty)}(x)$ meets $W_n$. The isotopy $\phi_t|_{F'_1}$ contracts $F'_1$ to $C_1$. We extend the deformation $\phi_t|_{F'_1\setminus C_1(f)}$ to the neighborhood $\bigcup_nW_n$. This extension can be defined as follows: for $x\in \bigcup_nW_n$, $t\in\R$ and $n\in I(x)$, there is a unique positive real number  $r(x,t,n)$ such that $\phi_{r(x,t,n)}(x)=\gamma(x)\cap \pi_n^{-1}(\phi_t(\pi_n(x)))$. Let $H:V_1\times\R\to X$ be the continuous map defined by $H(x,t)=\phi_{s(x,t)}(x)$, where
$$
  s(x,t)=\sum_{k\in I(x)}\lambda_k(x)\,r(x,t,k)\;.
$$ 
For $x\in \bigcup_nW_n$ and $t\in\R$, there exists $k_1,k_0\in I(x)$ such that $r(x,t,k_1)\leq s(x,t)\leq r(x,t,k_0)$. It is clear that there exists $\lim_{t\to \infty}\phi_{r(x,t,n)}(x)\in \overline{U_1}\subset\widetilde{U_1}$.  Let $p\in C_1$ and let $x\in F'_1\setminus C_1$ with $\omega(x)=p$. Since $\bigcup_nW_n$ is semisaturated and it is contained in $\bigcup_{i\in\N}\phi_{-i}(U_1)$, for all $r(x,t,k_1)< t' <r(x,t,k_0)$, $\phi_{t'}(x)\in \widetilde{U_1}$ for $t$ large enough. Therefore  $\lim_{t\to \infty}H(x,t)\in\widetilde{U_1}$ for all $x\in\bigcup_{n}V_n$. Then the measurable open subset $V'_1=\bigcup_{n}V_n$ is $\FF$-categorical (by a standard change of parameter). Finally,  if $\widetilde{U_1}$ is small enough, $U'_1=V'_1\cup \widetilde{U_1}$ is $\FF$-categorical by a telescopic argument  \cite{Hatcher}  and $F'_1\subset U'_1$.

This process can be continued inductively by taking $X_n=X\setminus (U'_0\cup\bigcup_{i=1}^{n-1}F'_i)$ and using the same trick to define $U'_n$, observing that $C_n(f)$ is the set of relative minima of $f|_{X_n}$.
\end{proof}

\begin{proof}[Proof of Theorem~\ref{t:LSMedFol}]
By Remark~\ref{r:critinf}, we can restrict the study to the case where $\Crit_\FF(f)$ is an isolated transversal. The previous proof also shows that $\Cat(\FF,\LB)$ is a lower bound for the sum of the measures of the critical sets. Since the critical sets form a partition of $\Crit_\FF(f)$, the proof is complete.
\end{proof}

\begin{ack}
This paper contains part of my PhD thesis, whose advisor is Prof. Jes\'{u}s A. \'{A}lvarez L\'{o}pez.
\end{ack}

\end{document}